\documentclass[10pt,twoside]{article}

\usepackage{enumerate,amsthm,amsmath,amssymb,graphicx,psfrag,xspace}
\usepackage{xspace}
\usepackage{comment}
\usepackage[%
   pdfpagemode=UseNone,
   bookmarks=true,
   colorlinks,
   linkcolor=blue,
   anchorcolor=blue,
   citecolor=blue,
   filecolor=blue,
   pagecolor=blue,
   urlcolor=blue
   ]{hyperref}

\def\eqref#1{(\ref{#1})}

\newcommand{\note}[1]%
{\noindent\centerline{\fbox{\parbox{.9\textwidth}{\textbf{#1}}}}}

\renewcommand{\P}{\mathcal{P}}
\newcommand{\PP}{\frak{P}}
\newcommand{\T}{\ensuremath{\mathcal T}}
\newcommand{\TT}{\ensuremath{\mathbb T}}
\newcommand{\IT}{I_{\T}}

\newcommand{\F}{\ensuremath{\mathcal F}}
\newcommand{\M}{\ensuremath{\mathcal M}}
\newcommand{\C}{\ensuremath{\mathcal C}}
\newcommand{\CC}{\ensuremath{\mathbb C}}

\newcommand{\norma}[1]{\left\| {#1} \right\|}
\newcommand{\seminorma}[1]{\left | {#1} \right |}

\newcommand{\R}{\mathbb R}
\newcommand{\V}{\mathbb V}
\newcommand{\NN}{\mathbb N}

\DeclareMathOperator{\dist}{dist}

\newcommand{\N}{\ensuremath{{\mathcal N}}}
\newcommand{\A}{\ensuremath{\mathbb A}}
\newcommand{\Asp}{{\ensuremath{\mathbb A_{s}^{p}}}}

\newcommand{\grad}{\nabla}

\DeclareMathOperator{\diam}{diam}

\newtheorem{theorem}{Theorem}[section]

\newtheorem{lemma}[theorem]{Lemma}

\newtheoremstyle{examplestyle}
  {6pt}
  {6pt}
  {}
  {}
  {\bfseries}
  {.}
  {1em}
  {}
\theoremstyle{examplestyle}
             
             \newtheorem{remark}[theorem]{Remark}
             \newtheorem{notation}[theorem]{Notation}
\newtheoremstyle{algorithmstyle}
  {6pt}
  {6pt}
  {\ttfamily}
  {}
  {\bfseries}
  {.}
  {1em}
  {}
\theoremstyle{algorithmstyle}
             
\title{Convergence rates for adaptive finite elements}
\author{Fernando D.\ Gaspoz \and Pedro Morin\\[10pt]
Consejo Nacional de Investigaciones Cient{\'\i}ficas y T\'ecnicas\\
and Universidad Nacional del Litoral\\
IMAL - G\"uemes 3450 - S3000GLN - Argentina\\
\texttt{fgaspoz}, \texttt{pmorin@santafe-conicet.gov.ar}
}

\begin{document}

\maketitle
\begin{abstract}
 In this article we prove that it is possible to construct, using newest-vertex bisection, meshes that equidistribute the error in $H^1$-norm, whenever the function to approximate can be decomposed as a sum of a regular part plus a singular part with singularities around a finite number of points. This decomposition is usual in regularity results of Partial Differential Equations (PDE). As a consequence, the meshes turn out to be quasi-optimal, and convergence rates for adaptive finite element methods (AFEM) using Lagrange finite elements of any polynomial degree are obtained.
\end{abstract}

\section{Introduction}

Adaptive procedures for the numerical solution of partial differential equations (PDE) started in the late 70's and are now standard tools in science and engineering. The ultimate purpose of adaptivity is to reduce the computational cost through the automatic construction of a sequence of meshes that would eventually equidistribute the approximation errors, leading to (quasi-)optimal meshes.  Adaptive methods for stationary problems usually consist of the loop
\begin{equation}
\label{SEMR}
 \text{\textsf{SOLVE} $\to$ \textsf{ESTIMATE} $\to$ \textsf{MARK} $\to$ \textsf{REFINE}}.
\end{equation}

Experience strongly suggests that, starting from a coarse mesh, such an iteration
converges within any prescribed error tolerance in a finite number of steps, and it does so in an optimal manner, provided the a posteriori error estimators are reliable and efficient. What is observed in fact, is that for a large class of problems and data, the solutions $u_\T$ and meshes $\T$ obtained with adaptive methods of the form~\eqref{SEMR} satisfy
\begin{equation}\label{optimal-decay}
 \| u - u_\T \|_{H^1} \le C (\#\T)^{-p/d},
\end{equation}
where $u$ denotes the exact solution, $p$ the polynomial degree of the finite element space over the mesh $\T$, and $d$ the dimension of the underlying space. This is the same error bound that is obtained with uniformly refined meshes for smooth (regular) solutions $u \in H^{p+1}$, by an application of classical interpolation estimates~\cite{Ciarlet}. The decay rate dictated by~\eqref{optimal-decay} ---which is also observed in practice for the so-called singular solutions belonging to $H^s(\Omega)$ for $s < 2$--- is usually called \emph{optimal error decay}. The precise goal of this paper is to show a broader family of functions for which this so-called optimal decay can be obtained when using adaptive methods. We will prove that this decay holds for functions that can be decomposed as a sum of a regular part plus singular terms, as described in classical regularity results for PDE~\cite{Gris1,Gris2,Petzoldt,Dauge}.

The first steps towards understanding the optimality of AFEM consisted of studying their convergence. An analysis of \eqref{SEMR} for linear, elliptic, and symmetric problems in 1d is presented in~\cite{BabVog}. The first multidimensional result is given in~\cite{Dor}, where it is proved that, after a pre-adaptation to data, \eqref{SEMR} reduces the error below any prescribed tolerance. Proper convergence without conditions on the initial grid is proved in~\cite{MNS}, requiring the so-called \emph{interior node property} and an additional marking step driven by \emph{data oscillation}.  The latter work was generalized in various directions.
Lately, convergence of adaptive methods with marking strategies other than D\"orfler's, for a large class of linear problems with different a posteriori error estimators, and without requiring the marking due to oscillation or the interior node property, was proved in~\cite{MSV}. The result only leads to asymptotic convergence without an error reduction in every step, which seems to be essential to prove optimality though~(see \cite{Stev1,CKNS}).

Regarding complexity, an important result for an algorithm which is very similar to~\eqref{SEMR}, is proved in~\cite{Stev1}. The proof relies on techniques first developed in~\cite{BDD} and new ideas. This result was later improved in several aspects in~\cite{CKNS}: the artificial assumptions of interior node and marking due to data oscillation were removed, and the result applies to more general elliptic equations.

When considering adaptive methods the notion of complexity differs from the previous one which was based on a uniform element size $h$. It is now defined in terms of the number of elements (or degrees of freedom) necessary to achieve a certain tolerance.

In order to be more specific at this point we need to introduce some notation.
Let us assume that we have a function $u \in H^1(\Omega)$, where $\Omega$ is a polygonal domain in $\R^2$ (polyhedral in $\R^3$), and $H^1(\Omega)$ denotes the Sobolev space of square integrable functions with square integrable weak derivatives of first order.

We consider an initial triangulation $\T_0$ of the domain $\Omega$ into simplices, and we let the \emph{admissible triangulations} be those obtained from $\T_0$ with newest-vertex bisection, either the iterative~\cite{Baensch} or the recursive~\cite{Kossaczky} version, without hanging nodes. For each admissible triangulation $\T$ we consider the Lagrange finite element space
\[
 \V_\T = \left\{ v \in H^1(\Omega) : v_{|T} \in \P^p, \forall T \in \T \right\},
\]
where, for $p \in \NN$, $\P^p$ denotes the space of polynomials of degree $\le p$. 
The best approximation error with complexity $N$, for $N\in\NN$, is defined as follows:
\[
 \sigma_N^p(u) = \min_{\T \in \TT_N} \inf_{v \in \V_\T} \| u - v \|_{H^1(\Omega)},
\]
where $\TT_N := \{ \T \text{ admissible} : (\#\T-\#\T_0) \le N \}$ that is, the minimum over $\T$ is taken over all admissible triangulations obtained with at most $N$ bisections. 
We now define, for $s > 0$ the approximation classes
\[
 \Asp = \left\{ v \in H^1(\Omega) : \exists C \text{ such that } \sigma_N^p(v) \le C N^{-s}, \forall N \in \NN\right\},
\]
or, equivalently,
\[
 \Asp = \left\{ v \in H^1(\Omega) : |v|_{\Asp} < \infty \right\}
\qquad\text{with}\qquad 
|v|_\Asp := \sup_{N\in\NN} \sigma_N^p(v) N^s .
\]

The first complexity results for adaptive finite element methods (AFEM) are presented in~\cite{BDD}, for an algorithm that needs coarsening, which seems not to be necessary, at least for symmetric elliptic problems. This, and the aforementioned papers on optimality of AFEM~\cite{Stev1,CKNS} study adaptive algorithms for approximating the solution $u$ to an elliptic partial differential equation. Essentially, the following fundamental result is proved: the adaptive algorithms generate a sequence $\{(\T_k, u_k)\}_{k\in\NN}$ of triangulations and finite element approximations $u_k \in \V_{\T_k}^1$ that satisfy the following:
\[
 \text{If}\quad u \in \A_s^1 \qquad\text{then}\qquad
 \| u - u_k \|_{H^1(\Omega)} \le \tilde C  (\# \T_k )^{-s}, \qquad \forall k \in \NN.
\]
That is, the sequence of triangulations and approximate solutions have a complexity with the same decay rate as the optimal ones.
The interesting aspect of those results is the fact that such a (quasi-)optimal approximation is obtained through a standard adaptive loop for the elliptic problem, without a priori knowledge of the exact solution, and with a number of operations proportional to the cardinality of the meshes. Notice that a simple minded approach to compute $\sigma_N^p(u)$ with precise knowledge of $u$ could lead to exponential work in terms of $N$.

The question ---already raised in~\cite{CKNS}--- that is still unanswered is what rate $s$ is to be expected in different situations. From the results just described it is clear that AFEM do a quasi-optimal job among all possible adaptive meshes. What we present in this article, is quantitative information about the convergence rate of AFEM.
In order to do so, we relate the membership of a function to an approximation class $\Asp$ with its regularity, proving rigorously, through the construction of specific meshes, that certain class of functions is contained in $\Asp$.

In~\cite{BDDP} an almost characterization of these classes is obtained, for the case $p=1$ in terms of Besov regularity for Lipschitz polygonal domains; the proof is based on an adaptive tree approximation algorithm. To illustrate the applicability of this result we just mention ---without giving too much detail--- that the Besov space $B_\tau^2(L_\tau(\Omega))$ is contained in $\A_{1/2}^1$ for all $\tau > 1$~\cite[Theorem 5.1]{BDDP}. 

The regularity of solutions to Poisson's problem on Lipschitz domains, in terms of Besov regularity is studied in~\cite{DD}. It is proved that for Poisson's equation $-\Delta u = f$ in a Lipschitz polygonal domain $\Omega \subset \R^2$, with homogeneous Dirichlet boundary values, $u \in B_\tau^2(L_\tau(\Omega))$ if $f \in H^1(\Omega)$. 

Combining these two results we obtain that $u \in \A_{1/2}^1$ if $f \in H^1$, but a stronger result holds. Using Grisvard's Sobolev regularity results~\cite{Gris1}, we have that only assuming $f \in L^2(\Omega)$, $u \in W_p^2(\Omega)$, that is, all derivatives of order up to two are in $L^p(\Omega)$, for all $1\le p < 4/3$. This, in turn implies that for all $1 < \tau < 4/3$, $u$ belongs to the Besov space $B_\tau^2(L_\tau (\Omega))$, and applying the result~\cite{BDDP} this implies $u \in \A_{1/2}^1$ under the sole assumption of $f\in L^2(\Omega)$.


The spirit of the results that we present in this article is a combination of~\cite{BDDP} and~\cite{DD}. However, our approach will not hinge upon regularity in Besov terms but rather upon a decomposition of the functions as a sum of a regular part plus singular terms, as stems from the classical regularity results for PDE like those stated in~\cite{Gris1,Gris2,Kellogg75,Kellogg92,Petzoldt,Dauge}. We obtain results for polygonal domains which are not necessarily Lipschitz (including slit domains) and we generalize to any polynomial degree $p\ge 1$; the proof is elementary, and does not make use of sophisticated theory of $L^q$ spaces for $q < 1$, as seems necessary in the approach of~\cite{BDDP}. Moreover, our result is directly applicable in some cases where the Besov regularity of the solutions to the PDE is not available, but instead, a descomposition into a regular plus a singular part is known to hold.

In~\cite{Gris1,Gris2} one can find some conditions on the element sizes relative to the distance to the points where the singularities are located, in order to obtain an error of order $N^{-1/2}$ when using linear elements in 2d. The difference between our result and those, is that we present an algorithm for \emph{constructing those meses using bisection}, and thus show that those meshes are attainable by an adaptive algorithm. Moreover, in view of the results in~\cite{Stev1,CKNS}, a consequence of our result is that the standard adaptive algorithms proposed there generate a sequence of meshes and discrete solutions $\left\{ \T_k, u_k \right\}_k$ satisfying $\| u - u_k \|_{H^1(\Omega)} \le C \left(\#\T_k\right)^{-p/d}$. A quantitative answer regarding convergence rates of adaptive finite element methods is thus obtained, for Lagrange finite elements of any polynomial degree $p \ge 1$.

The rest of the article is organized as follows. In section~\ref{S:main-result} we state the main result and present some applications to solutions of elliptic PDE in section~\ref{S:applications}. In section~\ref{S:construction} we propose an algorithm for constructing the desired mesh and prove some of its properties. We conclude the proof of the main result by bounding the error in section~\ref{S:error}.

\section{Main Result}\label{S:main-result}

 From now on, for any admissible triangulation $\T$ of the domain $\Omega$, we let $\V_\T$ denote the finite element space of continuous piecewise polynomials of degree $\le p$, where $p$ is a fixed positive integer. The following is the main result of this article, which states that a large family of functions, as those obtained when solving elliptic and other PDE, belong to $\A_{p/2}^p$. 

\begin{theorem}\label{T:main-result}
Let $ \Omega\subset\R^d $ be a polygonal ($d=2$) or polyhedral ($d=3$) domain, 
not necessarily Lipschitz, let $\T_0 $ be an initial triangulation of $\Omega $ and suppose that 
\begin{equation}\label{split}
  u = \sum_{i=0}^{N} u_i
\end{equation}
where:
\begin{itemize}
 \item  
$u_0\in H^1(\Omega)$, with $u_0\vert_{T} \in H^{p+1}(T)$, for all $T \in \T_0$;
 \item  for $i=1,2,\dots,N$, $u_i$ can be expressed in polar coordinates around $x_i$ as 
\[
u_i = c_i \, \big(\ln(r_i)\big)^{k_i} \, r_i^{\gamma_i} \, 
g_i( \overrightarrow{\theta_i} ) \, \chi_i ,
\]
where:
\begin{enumerate}
\item  $c_i$ are constants and $k_i$ are nonnegative integers.
\item  $\{x_i\}_{i=1}^{N} =: \mathcal{N}$ is a set of points in $\overline{\Omega}$, that are also vertices of $\T_0$;
\item $r_i$ denotes the distance to $x_i$, and: 
\begin{itemize}
 \item $\overrightarrow{\theta_i} = \theta_i \in [0,2\pi)$  is the angle coordinate of $x$ with respect to $x_i$ and a half line starting at $x_i$, when $d=2$;
 \item $\overrightarrow{\theta_i} = (\theta_i , \phi_i ) \in
   [0,2\pi)\times[0,\pi]$, where $\phi_i$ is the angle coordinate of
   $x$ with respect to $x_i$ and a half line $R$ starting at $x_i$,
   and letting $P$ denote the plane orthogonal to $R$ that contains
   $x_i$, $\theta_i$ is the angle coordinate of the projection of $x$
   on the plane a half line $S$ starting at $x_i$ contained into $P$,
   when $d=3$.
\end{itemize}
\item $\gamma_i$ are positive constants;
\item the functions $g_i$ satisfy the following assumptions depending on the dimension $d$:
\begin{itemize}
\item $g_i \in W_\infty^{1}(0,2\pi)$, satisfies the periodicity
  condition $g_i(0) = g_i(2\pi)$ and is piecewise $W_\infty^{p+1}$ in
  the following sense: there exists a partition $\PP_i$ of $[0,2\pi]$ into segments such that 
$g_i|_S   \in W^{p+1}_\infty(S) $ for all $S\in\PP_i$,
when  $d=2$;
\item $g_i \in W_\infty^{1}((0,2\pi)\times(0,\pi))$, satisfies the
  periodicity conditions $g_i(0,\phi_i) = g_i(2\pi,\phi_i)$,
  $0<\phi_i<\pi$, and $g_i(0,0)=g_i(\theta_i,0)$,
  $g_i(0,2\pi)=g_i(\theta_i,2\pi)$, $0<\theta_i<2\pi$, and is
  piecewise $W_\infty^{p+1}$ in the following sense: there exists a
  partition $\PP_i$ of $(0,2\pi)\times(0,\pi)$ into triangles such that $g_i|_S
  \in W^{p+1}_\infty(S) $ for all $S\in\PP_i$, when $d=3$;
\end{itemize}
\item $\chi_i$ are $C^\infty(\overline\Omega)$ cutoff functions;
\item the jumps of $\grad u_i$ (if any) are aligned with the edges of the initial mesh $\T_0$.
\end{enumerate}
\end{itemize}
 Then, for any given tolerance $ \varepsilon > 0$, there
exists a conforming triangulation $\T$, obtained by newest-vertex bisection, starting from $\T_0$ such that:
\begin{equation}\label{statement}
\inf_{u_{\T} \in \V_{\T}} \norma{u-u_{\T}}_{1, \Omega} \leq
\varepsilon \qquad
\text{and}\qquad \# \T  - \# \T_0 \leq \mathbf{C}_{u,\T_0} \, \frac{1}{\varepsilon^{d/p}},
\end{equation}
where $\mathbf{C}_{u,\T_0}$ depends on all the parameters that enter the definition of the singular part $ \sum_{i=1}^N u_i $, on  $\T_0$, and on $u$ through the broken seminorm 
$|u_0|_{H^{p+1}_{\T_0}(\Omega)} := \left(\sum_{T\in\T_0} \|D^{p+1} u_0\|^2_{L^2(T)}\right)^{1/2}$, but not on $\varepsilon$. 
Therefore $u \in \A_{p/d}^p$.
\end{theorem}

It is worth observing that, if $u$ satisfies the assumptions of the theorem, then we can only assure that $u \in H^{1+\epsilon}(\Omega)$ for all $0 < \epsilon < \min_{1\le i \le N} \gamma_i$. Uniform global refinements would only lead to $u \in \A_{\epsilon/d}^{p}$, but $\epsilon$ could be very small, and this rate is very pessimistic with respect to the one that can be obtained with adaptivity.

\begin{remark}\label{R:main-result}
 In order to shed some light on the assumptions of the theorem, we notice that they imply the following:
\begin{itemize}

 \item If we let $\gamma = \frac{\min_i \gamma_i}2$, 
we are able to control the singular terms through the following bound,
\begin{equation}\label{C-gamma}
C r_i^{\gamma} > \ln(r_i)^{k_i} r_i^{\gamma_i}. 
\end{equation}
and similar ones. They imply that, for each of the singular terms $u_i$, $i=1,2,\dots,N$, there exists a constant $C$, such that 
\begin{equation}\label{C-singamma}
|u_i| \le C r_i^\gamma, 
\quad |\grad u_i | \le C r_i^{\gamma-1}, 
\quad \text{and}  \quad |D^{p+1}u_i| \le C r_i^{\gamma-p-1},
\end{equation}
the last inequality holding only in the interior of the elements of $
\T_0 $, and thus also in the interior of any element of any refinement
of $\T_0$.  The constant $C$ depends on $c_i$, $k_i$, $\gamma_i$, 
the $ W^{p+1}_\infty $-norm of $\chi_i$, the
$ W^{1}_\infty $-norm of $g_i$, and the piecewise $W^{p+1}_\infty$-norm of $g_i$, that is, on the
$ W^{p+1}_\infty(S)$-norm
of $g_i$, for all $S \in \PP_i$.

The factor $ \frac{1}{2} $ in the definition of $\gamma$ is imposed to control the logarithmic terms. If all $k_i=0$, $ i=1,\dots,N $, then $\gamma$ could be chosen equal to $ \min_i \gamma_i $, and the same bounds would hold.

 \item if $\T$ is any refinement of $\T_0$, and $T\in\T$ with $ T \cap
   \N = \emptyset$ then $ u_i\vert_T \in H^{p+1}(T) $, $i=0,1,\dots,N$;

 \item since $p \ge 1$, and $d \le 3$, the Sobolev embedding theorem
   and the fact that $\gamma_i >0$ $i=1,2,\dots, N$ imply that each
   component $ u_i $, $ i=0,...,N $, is continuous in $
   \overline{\Omega} $, and consequently also $u$ is continuous;

\end{itemize}
This consequences of the assumptions are the main ingredients that will be used in the proof of our results below.
\end{remark}

\begin{notation}
From now on, the letter $C$ will denote a constant, not always equal,
depending on the given function $u$ of the assumption of
theorem~\ref{T:main-result}, through the $H^1(\Omega)$-norm of $u_0$,
the broken seminorm $|u_0|_{H^{p+1}_{\T_0}(\Omega)} :=
\left(\sum_{T\in\T_0} \|D^{p+1} u_0\|^2_{L^2(T)}\right)^{1/2}$, and
the parameters and functions defining the singular terms $u_i$,
$i=1,2,\dots,N$ of $u$ as in the second item of the previous remark.
We will reserve the notation $a\lesssim b$ to denote $a\le c\, b$ with
a constant $c$ depending only on shape regularity, or the geometry of
the domain. And $a \simeq b $ will indicate that $ a \lesssim b $ and
$b \lesssim a$.
\end{notation}

\section{Applications}\label{S:applications}

In this section we state two applications to elliptic PDE in two
dimensions in order to illustrate the applicability of our result.

\subsection{Poisson Equation}
Let $ \Omega$ be a polygonal domain in $ \mathbb{R}^2 $, not
necessarily Lipschitz. And let $u$ be the (weak) solution to
\begin{equation}\label{Poisson}
\begin{split}
- \Delta u &= f, \quad \text{in } \Omega, \\
         u &= 0, \quad \text{on } \partial\Omega,
\end{split}
\end{equation}
As a consequence of Theorem~3.1 in~\cite{Kellogg92} (see also~\cite{Dauge}, or Thm.~3.1 in~\cite{NVV}) it holds that if $f \in H^{p-1+\epsilon}(\Omega)$ for some $\epsilon > 0$, then $u$ can be written as in the assumptions in theorem~\ref{T:main-result}, where $\N = \{x_i\}_{i=1}^{N}$ is the set of vertices of $\Omega$, and $k_i = 0$, $i=1,2,\dots,N$.

In the case of $p=1$, $\epsilon$ can be taken to be zero, i.e.\ $f \in L^2(\Omega)$, the set $\N$ contains only the vertices of $\Omega$ with inner angle $\omega_i$ greater than $\pi$ ($c_i = 0$ for the other vertices), and $g_i(t) = \sin(\pi t/\omega_i )$ for all $i=1,2,\dots,N$.

In the case of $p > 1$, the set  $\N$ contains all the vertices of $\Omega$. In order to avoid the pathological cases where at least one inner angle $\alpha$ of $\Omega$ satisfies $\alpha p / \pi \in \NN$, we assume that $f \in H^{p-1+\epsilon}(\Omega)$ for some $\epsilon > 0$ instead of $H^{p-1}(\Omega)$, but this is not such a big restriction in practice. Moreover, this hypothesis can be weakened and ask that $f \in L^2(\Omega)$ and $f\vert_{T} \in H^{p-1+\epsilon}(T)$, for all $T \in \T_0$.

We conclude then that if $f \in H^{p-1+\epsilon}(\Omega)$ (piecewise over $\T_0$) then the solution $u$ to Poisson's equation~\eqref{Poisson} belongs to $\A_{p/2}^{p}$.

\subsection{Interface Problems for the Laplacian}
Let $ \Omega $ be a polygonal domain, not necessarily Lipschitz, that can be decomposed into disjoint subdomains $ \Omega_i , \ i=1,\dots,n_d $
with polygonal boundaries: $ \overline{\Omega} = \cup_{i=1}^{n_d} \overline{\Omega_i}$. We define the interface $ \Gamma =  \overline{( \cup_{i=1}^{n_d} ( \partial \Omega_i \setminus \partial \Omega ) )}$. 

Denote with $ a(x) = \sum_{i=1}^{n_d} a_i \chi_{\Omega_i}(x) $ the global weight function, which is constant and positive on each subdomain $ \Omega_i $.

We want to solve the following problem written in variational form:
\begin{equation}\label{inter-prob}
 \text{Find } u \in \V : \qquad \int_\Omega a \grad u \cdot \grad v \, dx 
          = \int_\Omega f v \, dx,
\qquad \forall v\in\V,
\end{equation}
where $ f \in L^2(\Omega) $, $\V = H^1_D(\Omega) = \left\{ v \in H^1(\Omega): v_{|\Gamma_D} = 0 \right\}$,  $\Gamma_D \subset \partial \Omega $ is the Dirichlet boundary. 
This problem  is usually called the \emph{interface problem for the Laplacian} and corresponds to the following strong form
\begin{equation*}\label{inter-prob2}
\begin{split}
- \nabla \cdot \big( a(x) \nabla u \big) &= f, \quad \text{in } \Omega_i,\quad i=1,2,\dots,n_d, \\
         u &= 0, \quad \text{on } \Gamma_D  \\ 
  \frac{\partial u}{\partial n} &= 0, \quad \text{on } \Gamma_N = \partial\Omega \setminus \Gamma_D, \\
   a_i \frac{\partial u_{|\Omega_i}}{\partial n_i} &= - a_j \frac{\partial u_{|\Omega_j}}{\partial n_j} \quad \text{on } \partial \Omega_i \cap \partial\Omega_j,
\end{split}
\end{equation*}
where $n$ denotes the outer unit normal to $\Omega$, and $n_i$ that of $\Omega_i$.

Following the original ideas from~\cite{Kellogg92}, Petzoldt proved (see Chapter~2 in~\cite{Petzoldt} and references therein) that the solution $u$ to~\eqref{inter-prob} satisfies the assumptions of theorem~\ref{T:main-result} for $p = 1$, if the mesh $\T_0$ matches the boundaries of the subdomains $\Omega_i$ and the points on $\partial\Omega$ where the boundary condition changes are vertices of $\T_0$. The points $x_\ell$ correspond to the vertices of the interface $\Gamma$, $\partial\Omega$, and to those points on $\partial\Omega$  where the boundary condition changes.

We conclude that if $ f \in L^2(\Omega) $ then by theorem~\ref{T:main-result} the  solution $u$ belongs to $\A_{1/2}^1$, and the optimal error decay is recovered.


It is worth mentioning that for certain singular points $x_\ell$, the value of $\gamma_\ell$ can be as close to zero as desired, depending on the values of $a(x)$ around $x_\ell$, providing very singular examples for the classical theory. In order to illustrate on this, we writedown the formulas derived by Kellogg~\cite{Kellogg75} to construct an exact solution of an elliptic problem with piecewise constant coefficients and vanishing right-hand side $f$; for the particular case $\Omega=(-1,1)^2$, $a = a_1$ in the first and third quadrants, and $a = a_2$ in the second and fourth quadrants.  An exact solution $u$ to~\eqref{inter-prob} for $f \equiv 0$ (and non-homogeneous Dirichlet boundary-values) is given in polar
coordinates by $u(r,\theta) = r^\gamma \mu(\theta)$, where
\begin{equation*}
\mu(\theta) = 
\begin{cases}
\cos((\pi/2-\sigma)\gamma) \cdot \cos((\theta-\pi/2+\rho)\gamma) 
\quad  & \text{if } 0 \le \theta \le \pi/2 \\
\cos(\rho \gamma) \cdot \cos ((\theta-\pi+\sigma) \gamma) 
       & \text{if } \pi/2 \le \theta \le \pi \\
\cos(\sigma \gamma) \cdot \cos((\theta - \pi - \rho) \gamma) 
       & \text{if } \pi \le \theta < 3\pi/2 \\
\cos((\pi/2 - \rho)\gamma) \cdot \cos((\theta - 3\pi/2 -
\sigma)\gamma) 
       & \text{if } 3\pi/2 \le \theta \le 2\pi
\end{cases}
\end{equation*}
and the numbers $\gamma$, $\rho$, $\sigma$ satisfy the 
nonlinear relations
\begin{equation}\label{ratios}
\left\{
\begin{aligned}
&R := {a_1}/{a_2} 
  = - \tan((\pi/2 - \sigma)\gamma) \cdot \cot(\rho\gamma)
\\
&1/R = - \tan(\rho \gamma) \cdot \cot(\sigma \gamma) 
\\
&R = -\tan(\sigma \gamma) \cdot \cot((\pi/2-\rho)\gamma)
\\
&0 < \gamma < 2 \\
&\max\{ 0, \pi \gamma - \pi\} < 2 \gamma \rho 
    < \min\{\pi \gamma, \pi \}   \\
&\max\{ 0, \pi - \pi \gamma \} < - 2 \gamma \sigma 
    < \min\{ \pi, 2 \pi - \pi \gamma \}.
\end{aligned}
\right.
\end{equation}
Choosing $\gamma = 0.1$, and solving~\eqref{ratios} for $R$, $\rho$
and $\sigma$ using Newton's method we obtain 
$R= a_1/a_2 \cong 161.4476$, 
$\rho = \pi/4$, 
$\sigma \cong -14.92256$.
A smaller $\gamma$ would lead to a larger ratio $R$, but in principle $\gamma$ may be as close to $0$ as desired.

This function $u$ belongs to the Sobolev space $H^{1+\gamma}(\Omega)$, and is thus \emph{barely in $H^1(\Omega)$}, but---according to our results---still in $\A_{p/2}^p$ for all $p \ge 1$. That is, an adaptive finite element approximation to a solution like this, using Lagrange finite elements of degree $p$ will lead to a sequence of meshes and discrete solutions  $\left\{ \T_k, u_k \right\}_k$ satisfying $\| u - u_k \|_{H^1(\Omega)} \le C \left(\#\T_k\right)^{-p/2}$. On the other hand, the Besov regularity of the solutions to~\eqref{inter-prob} is not well established, and thus the results of~\cite{BDDP} are not yet applicable to the interface problem for the laplacian. Until the Besov regularity of solutions to PDE is further developed, our result --- which is far from being a near characterization of the class of functions that can be approximated with optimal decay $ N^{-\frac{p}{d}} $ --- still provides a useful tool to investigate the convergence rate of AFEM for PDE.

\section{Construction}\label{S:construction}

From now on we assume that $u$ is as in the assumptions of theorem~\ref{T:main-result} and we will present an algorithm to construct via newest-vertex bisection a mesh fulfilling the properties stated in the theorem.

Before we introduce the algorithm we will present a heuristic idea with the ideal properties that the optimal mesh should have. This will motivate the precise definition of the algorithm, which is rather technical, but achieves with controlled complexity the goal of \emph{equidistribution}.

\subsection{Heuristic Idea}

Everything in this subsection will be heuristics, and is presented here---following the arguments in~\cite{Gris1,Liao-Noch,BAGMO}--- in order to motivate the properties that the optimal mesh should fulfill. The precise, rigorous proof will be given in the following sections, after the algorithm for constructing the mesh has been presented.

In order to introduce the basic idea consider the simplest case of a function $u$ written in polar coordinates as $u = r^\gamma \sin(\gamma \theta)$ on a two dimensional domain with a reentrant corner of inner angle $\pi/\gamma$ at the origin. Suppose that we approximate $u$ with continuous piecewise linear finite elements ($p=1$) on a mesh $\T$. The $H^1$ seminorm $|u-\IT u|_{1,T}$ of the error between $u$ and its Lagrange interpolant $\IT u$ on each element is bounded by $h \| D^2 u \|_{L^2(T)}$ if $0 \notin T$ and by $\|Du\|_{L^2(T)}$ if $0 \in T$. These quantities (squared) also satisfy the following:
\[
 \begin{aligned}
 h_T^2 \| D^2 u \|_{L^2(T)}^2 &\cong h_T^2  r_T^{2(\gamma-2)} |T| \cong h_T^4 r_T^{2(\gamma-2)}, \qquad&\text{if }0 \notin T, \\
 \|Du\|_{L^2(T)}^2 &\cong \int_0^{h_T} r^{2(\gamma-1)} \, r \, dr \cong h_T^{2\gamma} , \qquad&\text{if } 0 \in T,
\end{aligned}
\]
where $r_T$ denotes the distance of $T$ to the origin and $h_T := |T|^{1/2} \cong \diam(T)$. 
In order to achieve the equidistribution of the local error bounds we then require for the mesh $\T$ that, given a small parameter $h > 0$, the elements satisfy
\[
  h_T^4 r_T^{2(\gamma-2)} \cong h^{2\gamma}, \quad\text{if } 0\notin T, \qquad
  \text{and}\qquad h_T \cong h, \quad\text{if } 0 \in T.
\]

Suppose now that this goal is achievable. More precisely, we can classify the elements into rings at dyadic distance to the origin, by defining
\[
D_k = \left\{ T \in \T : 2^{-k-1} \leq r_T < 2^{-k} \right\},
\]
for $k \in\NN$, $k < K:=\lfloor\log_2(1/h)\rfloor$, and $D_{K} = \left\{ T \in \T: r_T < 2^{-K} \right\}$.

Then, on the one hand, the elements $T \in D_k$, have size $|T| = h_T^2 \cong h^{\gamma} r_T^{-(\gamma-2)} \cong h^{\gamma} 2^{k(\gamma-2)}$, and thus $\# D_k \cong \frac{2^{-2k} }{ h^{\gamma} 2^{k(\gamma-2)}} = h^{-\gamma}2^{-k\gamma}$ which implies that 
\[
 \#\T \cong \sum_{k \le K} \# D_k \cong h^{-\gamma} \sum_k 2^{-k\gamma} \cong  h^{-\gamma}.
\]
On the other hand, the error satisfies
\[
 |u - u_h |_{1,\Omega}^2 \cong \#\T \, h^{2\gamma} \cong h^{-\gamma} h^{2\gamma}
= h^\gamma \cong \left(\# \T\right)^{-1}.
\]
And this finally implies that $ |u - u_h |_{1,\Omega} \lesssim \left( \# \T\right)^{-1/2}$, and thus $u \in \A_{1/2}^1$.

In the case $d=3$ if $u$ has a singularity like $r^{\gamma}$ as in the previous example, the bound $ |u - u_h |_{1,\Omega} \lesssim \left( \# \T\right)^{-1/3}$, would be obtained if 
\[
  h_T^5 r_T^{2(\gamma-2)} \cong h^{2\gamma +1}, \quad\text{if } 0\notin T, \qquad
  \text{and}\qquad h_T \cong h, \quad\text{if } 0 \in T.
\]

\subsection{Algorithm}

In this section we will introduce the algorithm that will achieve using newest-vertex bisection, a mesh with the precise grading stated in the previous subsection, generalized to polynomials of degree $p$. 

From now on we will use the notation
\[
r_X = \min_{x_i \in \mathcal{N}}  \dist(x_i , X) 
\]
defined for $X$ compact, typically $X$ is a triangle $T$ or a point $x$, where $\N$ denotes the finite set where the singularities are located (as in the assumptions of theorem~\ref{T:main-result}).

We choose and fix $ \gamma = \frac{\min_i \gamma_i}{2} $. This choice
allows us to bound the singular terms as in~(\ref{C-singamma}).

Let $\T_0$ be the given initial mesh and let $\delta>0$ be a small parameter so that
$ \# \T_0 \leq \delta^{-d} $. Later $\delta$ will be chosen such that $\delta^{p} \approx \varepsilon$, where $\varepsilon$ is the error to be achieved between $ u $ and $ u_{\T} $, a discrete approximation to $u$ in $ \V_{\T} $, and $ \T $ the mesh generated by the algorithm (see proof of theorem~\ref{T:main-result} in section~\ref{s:main-proof}). Now let $ K \in \mathbb{N} $ be such that
\begin{equation}\label{K}
 2^{-\frac{(K+1) ( 2 \gamma + d -2)}{2p+d}}\leq \delta < 2^{-\frac{K ( 2 \gamma + d -2)}{2p+d}}. 
 \end{equation}

Denoting for any element $T$ the elementsize by $h_T = |T|^{1/d}$, the constructive algorithm reads:

\

$\T_{0,0}^{c} \leftarrow \T_0$

$ j = 0 $

\texttt{\% initial (global) refinement to control the error of $u_0$}

\texttt{\% FIRST LOOP}

\texttt{do}

\ \ $ \mathcal{M}_{0,j} = \{ T\in \T_{0,j}^{c} : h_T >
\delta \} $

\ \ $ \T_{0,j+1} \leftarrow $\texttt{ refine}$ (\T_{0,j}^c , \mathcal{M}_{0,j}) $

\ \ $ \T_{0, j + 1}^{c} \leftarrow $\texttt{ complete}$ (\T_{0 , j + 1}) $

\ \ $ j \leftarrow j + 1 $

\texttt{until} $ \mathcal{M}_{0, j-1} = \emptyset $

$ J = j $

$ \T_{1}^{c} \leftarrow \T_{0,j}^{c} $

$ \ell = 1 $

\texttt{\% Selective refinement according to distance to singularities}

\texttt{\% SECOND LOOP}

\texttt{while} $( \ell < d (K+1) )$

\ \ $ \Omega_{\ell} = \bigcup \{ T \mid T\in \T_{\ell}^c \ \wedge
\ r_T \leq 2^{-\frac{\ell}{d}} \} $

\ \ $ \M_{\ell} = \{ T \subset \Omega_{\ell} : \ h_T >
\delta \, 2^{\frac{2 \ell ( \gamma - p -1 )}{d(2p+d)}} \}$

\ \ $ \T_{\ell + 1} \leftarrow $\texttt{ refine}$ ( \T_{\ell}^c ,
\mathcal{M}_{\ell} ) $

\ \ $ \T_{\ell + 1}^{c} \leftarrow $\texttt{ complete}$ (\T_{\ell + 1}) $

\ \ $ \ell \leftarrow  \ell + 1 $

\texttt{end}

\ 

The algorithm makes use of two routines that need further explanation. The first one, 
\[
 \T_{\text{new}} \leftarrow \texttt{ refine}(\T_{\text{old}}, \M)
\]
receives a mesh $\T_{\text{old}}$, usually admissible, and a set $\M $ of \emph{marked} elements from $\T_{\text{old}}$. It returns a new mesh $\T_{\text{new}}$ that is obtained after bisecting once the marked elements according to the newest-vertex bisection rule. The new mesh is not necessarily admissible (it may have hanging nodes), but it clearly holds that
\[
 \# \T_{\text{new}} = \# \T_{\text{old}} + \# \M,
\]
i.e., $ \# \T_{\text{new}} - \# \T_{\text{old}} = \# \M$.

The second routine that is used,
\[
 \T^c \leftarrow \texttt{ complete}(\T)
\]
receives a mesh $\T$ that is not necessarily admissible, and returns a new mesh $\T^c$ which is made admissible by refining the least amount of necessary elements, again by the newest-vertex bisection rule. The study of complexity of this routine is not as easy as that of the previous one, and it is not true that there exists a constant $\mathcal{C}$ such that 
\[
 \# \T_{\ell+1}^c \le \# \T_\ell^c + \mathcal{C} \big( \# \T_{\ell+1} - \# \T_\ell^c \big).
\]
The complexity result that holds---regarding the \emph{spreading of refinement} implied by the completion algorithm---is the following one, which is a little bit weaker, but fundamental and sufficient for the purposes of studying optimality of AFEM:

\begin{theorem}\label{T:bdd}
Let $\T_0 = \T_0^c$ be an initial admissible mesh of a polygonal (polyhedral) domain $\Omega$ in $ \mathbb{R}^2 $ ( $ \mathbb{R}^3 $), whose elements edges are properly flagged in the sense that whenever an interior edge is a refinement edge, it is the common refinement edge for all adjacent elements. If the sequence $\{ \T_\ell^c \}_{\ell \ge 1}$ is obtained by subsequent calls to:
\begin{align*}
    \T_{\ell + 1} &\leftarrow \text{\tt refine} ( \T_{\ell}^c ) \\
\T_{\ell + 1}^{c} &\leftarrow \text{\tt complete} (\T_{\ell + 1}), 
\end{align*}
then for $k \ge 1$ we have that
\begin{equation*}
\# \T_{k}^c - \# \T_0 
\leq \mathcal{C} \bigg( \sum_{\ell = 1}^{k} (\# \T_{\ell +1} - \# \T_{\ell}^{c}) \bigg),
\end{equation*}
where $\C$ is a constant depending only on $\T_0$.
\end{theorem}
This result was first proved in~\cite{BDD} for triangles, and later extended to simplicial meshes of any dimension in~\cite{Stev2}.

As a consequence of this we have that if now $\T_{0,j}$, $\T_{0,j}^c$ $\T_\ell$, $\T_\ell^c$ are the meshes obtained by our algorithm it holds that
\begin{equation}\label{bdd}
\# \T_{d(K+1)}^c - \# \T_0 
\leq \mathcal{C} \bigg( \sum_{\ell = 1}^{d(K+1)-1}
( \# \T_{\ell +1} - \# \T_{\ell}^{c} ) + \sum_{j = 0}^{J-1} ( \#
\T_{0,j +1} - \# \T_{0,j}^{c} ) \bigg).
\end{equation}

\begin{remark}
 
Before proceeding to the proof of our result, some remarks are in order:

\begin{itemize}
 \item The idea of the algorithm is to achieve an equidistribution of the error following the heuristics stated in the previous section. Since the refinement is stronger closer to the singularity points, our approach considers a sequence of regions $\Omega_\ell$ around them with  geometrically decreasing radii given by $2^{-\frac{\ell}{d}}$. The denominator $d$ in the exponent is related to the fact that we perform only one bisection to marked elements in \texttt{refine}, and $d$ are necessary to achieve a halving of $h_T$.

 \item The algorithm does not take into account the different sizes of the powers $\gamma_i$, it just looks at a worst case scenario taking a unified value $\gamma = \frac{ \min_i \gamma_i}{2}$. As we will see later, the property $\gamma > 0$ is the only one used in the proof. In the same manner, the distance to the singularity points $x_i$ is unified by taking the minimum distance symbolized by $r_T$. It may look that the simplification introduced by this \emph{unification} will lead to sub-optimal meshes, and it is true that the constant $\mathbf{C}_{u,\T_0}$ in~\eqref{statement} may be bigger than necessary with this approach. But this is an a priori approach where we want to show the \emph{membership} of certain functions to the spaces $\A_{p/d}^{p}$, not caring about the size of their norm.

 \item If an efficient construction of the mesh is desired, the algorithm could be improved by marking separately according to the different strengths of the singularities. This would lead to a better constant $\mathbf{C}_{u,\T_0}$, but the overall theoretical result will be the same. We decided to present this unified approach for the ease of presentation.

 \item One important property of the newest-vertex bisection rule is that it leads to a sequence of meshes with a uniformly bounded shape-regularity constant, which depends only on that from the initial mesh $\T_0$ and the new-vertex flagging of the initial mesh. We thus have that all the meshes $\T_\ell^c$ obtained by the application of our algorithm are shape-regular with a uniform constant. 
\end{itemize}

\end{remark}

\subsection{Properties of the Algorithm}

In this section we will bound through a series of lemmas the complexity of the resulting mesh $\T_{d(K+1)}^c$, and in the next section we will relate this complexity to the error of the best approximation to $u$ through finite element functions over $\T_{d(K+1)}^c$. 

The following lemma is related to the termination of the first loop of the algorithm in a finite number of steps, and to a control on the number of elements added.
The termination of the second loop is straightforward, since it is just a \emph{for} loop in disguise.

\begin{lemma}\label{1loop}
The first loop of the algorithm terminates after $J$ iterations, with 
$J \leq \log_{2}\left(\frac{\max_{T \in \T_0} |T|}{\delta^{d}}\right) + 1$
and there exists a constant $\CC_1=2|\Omega|$ such that:
\begin{equation}\label{complexity-initial-loop}
\sum_{j=0}^{J-1} \big(\# \T_{0,j+1} - \# \T_{0,j}^{c}\big) \leq
\CC_1 \delta^{-d}.
\end{equation}
This implies that for all $ T \in \T_{1}^{c} $, $ |T| < \delta^{d}  $.
\end{lemma}
\begin{proof}
Observe first that if one bisects an element $T \in \T_0$, $J$ times with $J \ge \log_{2}\left(\frac{\max_{T \in \T_0} |T|}{\delta^{d}}\right) + 1 $, then the measure of the resulting sub-elements will be strictly less than $\delta^{d}$, and the marking step will not mark them anymore. This
proves the first part of the statement.

In order to prove the bound~(\ref{complexity-initial-loop}) we define, for $i \ge 0$
\[
 \F_i = \bigg\{ T \mid T \in \bigcup_{k}
\T_{0,k}^c \,\wedge\, 2^{i} \delta^{d} \leq |T| < 2^{i+1}
\delta^{d} \bigg\}.
\]
It is easy to see that even though $\F_i$ contains elements belonging to different meshes, they do not overlap, and then:
\[
| \Omega | \geq \sum_{T \in \F_i } |T|  \geq
\sum_{T \in \F_i } \delta^{d} 2^{i} = \delta^{d} 2^{i} (\# \F_i)  ,
\]
which implies that $\# \F_i \leq  | \Omega | \delta^{-d}
2^{-i} $ .

Now, applying these estimates,  and using that
\[
\bigcup_{i=0}^{\infty}
\F_i = \{ T \mid T \in \bigcup_{k=0}^{J} \T_{0,k} \
\wedge \ |T| \geq \delta^{d} \} = \bigcup_{j=0}^{J-1}
\mathcal{M}_{0,j},
\]
we obtain that
\begin{align*}
\sum_{j=0}^{J-1} (\# \T_{0,j+1} - \# \T_{0,j}^{c}) =
\sum_{j=0}^{J-1} \# \mathcal{M}_{0,j}= \sum_{i=0}^{\infty} \#
\F_i \leq 2|\Omega| \delta^{-d},
\end{align*}
and the claim is proved.
\end{proof}

\begin{remark}
 This proof is a little complicated due to the way the algorithm was proposed in order to take into account any previous grading of the mesh. Observe that in the first loop we do not refine all the elements, but only those which are bigger than the threshold $\delta$, instead of doing just uniform global refinements. If we did this, the proof would be simpler, but the number of elements in $\T_1^c$ would be unnecessarily bigger.
\end{remark}

The following lemma is just an observation of the fact that if a point $z$ is a vertex of a shape-regular triangulation, then the distance of the elements to $z$ is an upper bound to the diameter of the element, unless of course the distance is zero. This means that the diameter of the elements can grow at most linearly with the distance to a point. 

\begin{lemma}\label{dist}
Let $ \T $ be a regular mesh such that $z$ is a node, then $
\forall T \in \T $ with $ \dist(z,T) \neq 0 $ we have that $ |T |
\lesssim \dist(z,T)^d $, or $h_T \lesssim \dist(z,T)$.
\end{lemma}

This result may be familiar to some practitioners, but it is not completely obvious. A stronger result was proved in~\cite[Lemma~5.1]{NPV}, but we decided to include its proof here for the sake of completeness.

\begin{proof}
Let $ T $ be an element of $ \T $  and let us define  $ \omega_T = \bigcup \{
\bar{T} \mid  \bar{T} \in \T \ \wedge \ T\cap \bar{T} \neq
\emptyset \}$. If $ z \notin \omega_T $ then by shape regularity, $\dist(z,T) \geq c h_T$. If $ z \in \omega_T \backslash T $, then $z$ is a vertex of a neighboring element $T'$ and thus $ \dist(z,T) \approx h_{T'} \approx h_T $.
\end{proof}

The next result implies that the desired grading of the mesh was achieved by the algorithm.

\begin{lemma}\label{size}
Let $ \T = \T_{d(K+1)}^{c} $, then for $0 \le \ell \le d(K+1)$ the following property holds:
\[
T \in \T \quad\text{and}\quad r_T<  2^{-\frac{\ell }{d}} 
\qquad \implies\qquad
|T| < \delta^{d} 2^{\frac{2 \ell (\gamma - p - 1)}{2p+d}}.
\]
\end{lemma}
\begin{proof}
We first claim that for each $ 0 \leq \ell < d(K+1) $, the
following holds for the intermediate triangulations $\T_{\ell+1}$:
\begin{equation}\label{g-size}
T \in \T_{\ell + 1}^{c} \quad\text{and}\quad r_T <
2^{-\frac{\ell}{d}} \qquad\implies\qquad |T| < \delta^{d}
2^{\frac{2\ell (\gamma - p - 1)}{2p+d}}.
\end{equation}
We prove this by induction on $\ell$: By lemma~\ref{1loop} it holds
for $ \ell = 0 $. Before proceeding, observe that: if $ T' \in \T_{\ell}^{c} $ and $
T \in \T_{k}^{c} $ with $ k > \ell $:
\begin{equation}\label{prop-dist}
T \subset T'  \qquad\implies\qquad r_T \geq r_{T'}.
\end{equation}
Suppose now that~\eqref{g-size} holds for $ \ell $ and let us prove it
for $ \ell+1 $. If $ T \in \T_{\ell + 2}^{c} $ and $ r_T < 2^{-\frac{\ell+1}{d}} $, there exist $ T' \in \T_{\ell + 1}^{c} $ such that $ T \subset T' $, whence $ r_{T'} <
2^{-\frac{\ell}{d}} $, and by the inductive assumption $|T'| < \delta^{d}
2^{\frac{2 \ell (\gamma - p - 1)}{2p+d}} $. Now, if already $|T'| < \delta^{d} 2^{\frac{2 (\ell+1) (\gamma - p - 1)}{2p+d}} $ then the results
holds because $|T| \leq |T'|$. Otherwise, $ T' \in
\mathcal{M}_{\ell +1} $ and we have that 
\[
 |T| \leq \frac{1}{2}
|T'| < \frac{\delta^{d} 2^{\frac{2 \ell (\gamma - p - 1)}{2p+d}}}{2} < \delta^{d}  2^{\frac{2(\ell+1) (\gamma -
p - 1)}{2p+d}} 
\]
because $ \gamma > 0 $ and $d \geq 2$. Thus~\eqref{g-size} is proved for
$ \ell +1 $.

We now proceed to prove the claim of the lemma: Let $ T \in \T $ such that $
r_T < 2^{-\frac{\ell}{d}} $, then there exist $ T' \supset
T $, $ T' \in \T_{\ell +1}^{c} $ and then by \eqref{prop-dist}, $
r_{T'} < 2^{-\frac{\ell}{d}} $, and by~\eqref{g-size}, $ |T|
\leq |T'| < \delta^{d} 2^{\frac{2\ell (\gamma - p - 1)}{2p+d}} $.
\end{proof}

The claim of the previous lemma could have been achieved by simple uniform refinement, but this would have destroyed the complexity of the mesh. The next lemma shows that the number of marked elements in each iteration is reasonably bounded in a way that the overall complexity of the final mesh is under control.

\begin{lemma}\label{2loop}
There exists a constant $\mathbb{C}_2$, depending only on shape regularity, such that for $ 1 \leq \ell
< d(K+1) $:
\begin{equation}
\# \M_\ell = \# \T_{\ell + 1} - \# \T_{\ell}^{c} 
\leq \mathbb{C}_2 \, \delta^{-d} 2^{-\frac{\ell (2\gamma + d -2)}{2p+d}} .
\end{equation}
\end{lemma}
\begin{proof}
Recall that in the algorithm we define $ \Omega_{\ell} = \bigcup
\{ T \mid T\in \T_{\ell}^c \ \wedge \ r_T \leq
2^{-\frac{\ell}{d}} \} $, and since $\T_{\ell+1}$ is obtained from $\T_\ell^c$ by refinement only, we have that $\Omega_\ell = \bigcup
\{ T \mid T\in \T_{\ell+1} : T \subset \Omega_\ell \} $, whence
\[
| \Omega_{\ell} | 
= \sum_{T \in \T_{\ell +1},\ T \subset \Omega_{\ell} } h_T^d 
= \sum_{T \in \T_{\ell +1} \backslash \T_{\ell}^c } h_T^d 
   + \sum_{\substack{T \in \T_{\ell +1} \cap \T_{\ell}^c\\ \ T \subset \Omega_{\ell} }} h_T^d 
\geq \sum_{T \in \T_{\ell +1} \backslash \T_{\ell}^c } h_T^d .
\]
But if $T \in \T_{\ell +1} \backslash \T_{\ell}^c$, then $T$ is half of an element $T' \in \M_\ell$, and thus by the definition of $\M_\ell$ in the algorithm,
\[
 2 \, h_T^d = h_{T'}^d \ge \delta^{d} 2^{\frac{2\ell (\gamma -p-1)}{2p+d}},
\]
which in turn implies that
\[
| \Omega_{\ell} | \geq
\frac{\delta^{d}  2^{\frac{2\ell ( \gamma - p -1 )}{2p+d}}  }{2} ( \#
\T_{\ell +1} - \# \T_{\ell}^c ) \\
=  \frac{ \delta^{d} \, 2^{-\ell}\, 2^{\frac{\ell (2\gamma + d -2)}{2p+d}}}{2} ( \#
\T_{\ell +1} - \# \T_{\ell}^c ).
\]
By lemma \ref{dist}  we have that $| \Omega_{\ell} | \leq \CC 2^{-\ell} $ and then:
\[
\# \T_{\ell + 1} - \# \T_{\ell}^{c}  
 \leq 2 \, | \Omega_{\ell} | \, 2^{\ell} \,
\delta^{-d} 2^{-\frac{\ell (2\gamma + d -2)}{2p+d}} 
\leq \mathbb{C}_2 \, \delta^{-d} 2^{-\frac{\ell (2\gamma + d -2)}{2p+d}} ,
\]
and the lemma is proved.
\end{proof}

The next lemma makes use of the complexity result~(\ref{bdd}) of the completion procedure for the newest-vertex bisection rule, to bound the complexity of the final mesh.

\begin{lemma}\label{complexity}
There exists a constant $\mathbb{C}_3$, depending only on shape regularity, the polynomial degree $p$, the dimension $d$, the function $u$ through $\gamma$, and $\T_0$, such that:
\begin{equation}
\# \T_{d(K + 1)}^{c} - \# \T_{0}  \leq \mathbb{C}_3 \, \delta^{-d}.
\end{equation}
\end{lemma}
\begin{proof}
Using~(\ref{bdd}), lemmas~\ref{1loop} and~\ref{2loop} we have that
\begin{align*}
\# \T_{d(K + 1)}^{c} - \# \T_{0} 
&\leq \C \bigg( \sum_{\ell =1}^{d(K+1)-1} ( \# \T_{\ell + 1} - \# \T_{\ell}^c )
+ \sum_{j=0}^{J-1} (\# \T_{0,j+1} - \# \T_{0,j}^{c}) \bigg) \\
&\leq \C \bigg( \sum_{\ell =1}^{d(K+1)-1} 
       \CC_2  \delta^{-d} 2^{\frac{- \ell (2\gamma +d-2)}{2p+d}} 
     + \CC_1  \delta^{-d} \bigg)
\\ &\leq \C \delta^{-d}
\bigg(\CC_2 \sum_{\ell =1}^{\infty}   2^{\frac{- \ell (2\gamma +d-2)}{2p+d}} + \CC_1 \bigg). 
\end{align*}
Since $\gamma > 0$ the sum $ \sum_{\ell =1}^{\infty}   2^{\frac{- \ell (2\gamma +d-2)}{2p+d}}$ is finite, and the claim follows taking $\CC_3 = \C \bigg(\CC_2 \sum_{\ell =1}^{\infty}   2^{\frac{- \ell (2\gamma +d-2)}{2p+d}} + \CC_1 \bigg)$.
\end{proof}

\section{Error}\label{S:error}

In this section we bound the best error with finite element functions in terms of the complexity of the mesh:

\begin{theorem}\label{T:error}
There exist two constants $\mathbf{A}_1$, $\mathbf{A}_2$, that may depend on $u$ through the broken seminorm 
$|u_0|_{H^{p+1}_{\T_0}(\Omega)} := \left(\sum_{T\in\T_0} \|D^{p+1} u_0\|^2_{L^2(T)}\right)^{1/2}$, $c_i$, $k_i$, $\gamma_i$, the $\|\chi_i\|_{W^{p+1}_\infty(\Omega)}$, $\|g_i\|_{ W^{1}_\infty(\Omega)}$, and the $ W^{p+1}_\infty(S) $-norm of $g_i$, $ S \in \PP_i $, $i=1,\dots,N$, the polynomial degree $p$, the dimension $d$, shape regularity and $\T_0$, but otherwise independent of $K$ and $ \delta $, such that, if $ \T = \T_{d(K+1)}^{c} $, then
\begin{align}
    \inf_{u_{\T} \in \V_{\T}} \norma{u-u_{\T}}_{1, \Omega} 
            &\leq \mathbf{A}_1 \delta^{p},   \\
    \inf_{u_{\T} \in \V_{\T}} \norma{u-u_{\T}}_{1, \Omega} 
            &\leq \mathbf{A}_2 ( \# \T - \# \T_0)^{-\frac{p}{d}}.
\end{align}
\end{theorem}

In order to prove this theorem we will consider the regular part $u_0$ of $u$ and the singular part given by $\sum_{i=1}^ N u_i$.

Throughout this section we will use the Lagrange interpolator $\IT u_i$ of $u_i$, which is the finite element function that coincides with $u_i$ at all the nodes, and is well defined for each $i=0,1,\dots,N$, since by the assumptions of theorem~\ref{T:bdd}, all the $u_i$ functions are continuous in $ \overline{\Omega} $; see remark~\ref{R:main-result}.

\subsection{Estimation of the Regular Part}

\begin{theorem}\label{T:est-reg}
There exist two constants $\mathbb{C}_4$, $\mathbb{C}_5$, depending on the broken seminorm 
$|u_0|_{H^{p+1}_{\T_0}(\Omega)} := \left(\sum_{T\in\T_0} \|D^{p+1} u_0\|^2_{L^2(T)}\right)^{1/2}$, the polynomial degree $p$, shape regularity and $\T_0$, but otherwise independent of $K$ and $\delta$, such that, if $ \T = \T_{d(K+1)}^{c} $, then
\begin{align*}
  \seminorma{u_0-\IT u_0}_{1, \Omega} &\leq \mathbb{C}_4 \delta^{p} ,
\\
  \seminorma{u_0-I_{\T}u_0}_{1, \Omega} &\leq \mathbb{C}_5 ( \# \T - \# \T_0)^{-\frac{p}{d}} .
\end{align*}
\end{theorem}
\begin{proof}
Since $u_0\vert_T \in H^{p+1}(T)$ for all $T\in \T_0$, and $\T$ was obtained only by refinement, $u_0\vert_T \in H^{p+1}(T)$ for all $T\in \T$, and standard interpolation estimates~\cite{Ciarlet} yield
\begin{align*}
|u_0 - I_{\T} u_0|_{1, \Omega}^2 &= \sum_{T \in \T} |u_0 -
I_{\T}u_0|_{1,T}^2 \lesssim \sum_{T \in \T} h_T^{2p} \|D^{p+1} u_0
\|_{L^2(T)}^2 
\\ &\leq \delta^{2p} \left| u_0 \right|_{H^{p+1}_{\T_0}(\Omega)}^2 ,
\end{align*}
where the last inequality is a consequence of lemma~\ref{1loop} and the first loop of the algorithm.

Then, by lemma~\ref{complexity}
\[ 
|u_0 - I_{\T} u_0|_{1,\Omega} \le \CC_4 \delta^{p} 
\le \CC_5 (\# \T  -\# \T_0 )^{-\frac{p}{d}},
\] 
and the theorem is proved.
\end{proof}

\subsection{Estimation of the Singular Part}

Throughout this section, we will denote with $u$ one of the singular terms $u_i$ defining $u$ in~(\ref{split}), that is, it is defined in polar coordinates around a point $x_i$ in $\overline{\Omega}$ as
\begin{equation}\label{def-u-sing}
u = c_i \, \big(\ln(r_i)\big)^{k_i} \, r_i^{\gamma_i} \, g_i( \overrightarrow{\theta_i} ) \, \chi_i ,
\end{equation}
for some $i=1,2,\dots,N$ and $c_i$, $r_i$, $k_i$, $\gamma_i$, $g_i$, $\overrightarrow{\theta_i}$, $\chi_i$ as in the assumptions of theorem~\ref{T:main-result}.

The three bounds of~(\ref{C-singamma}) are the only features of $u$ that will be used in the proof of the following theorem.


\begin{theorem}\label{T:est-sing}
There exist two constants $\mathbb{C}_6$, $\mathbb{C}_7$, that depend on the parameters defining $u$ in~\eqref{def-u-sing}, shape regularity and $\T_0$, but otherwise independent of $K$ and $\delta$, such that, if $ \T = \T_{d(K+1)}^{c} $, then
\begin{align*}
 \seminorma{u-I_{\T}u}_{1, \Omega} &\leq \mathbb{C}_6 \delta^{d} ,
\\
 \seminorma{u-I_{\T}u}_{1, \Omega} &\leq \mathbb{C}_7 ( \# \T - \# \T_0 )^{-\frac{p}{d}} .
\end{align*}
\end{theorem}

\begin{proof}
Let $D_{\ell}= \bigcup \{ T \mid T \in \T \ \wedge \
2^{-\frac{\ell + 1}{d}} < \dist(x_i , T) \leq 2^{-\frac{\ell}{d}}
\} $ for $ 0 \leq \ell < d(K+1) $ and $ D_{d(K+1)}= \bigcup \{ T
\mid T \in \T \ \wedge \ \dist(x_i , T) \leq 2^{-(K+1)} \} $. Then
we obtain:
\begin{align}\notag
\seminorma{u-I_{\T} u}^2_{1, \Omega} &=  \sum_{T \in
\T}\seminorma{u - I_{\T} u}^2_{1,T}
\\
\label{main}
&= \sum_{\ell =0}^{d(K+1)-1} \sum_{T \subset D_{\ell}} \seminorma{u -
I_{\T} u}^2_{1,T} + \seminorma{u - I_{\T} u}^2_{1,D_{d(K+1)}} .
\end{align}
The second term in~(\ref{main}) can be bounded as follows:
\begin{align*}
 \seminorma{u - I_{\T} u}^2_{1, D_{d(K+1)}} &\leq
\seminorma{u}^2_{1, D_{d(K+1)}} + \seminorma{I_{\T} u}^2_{1, D_{d(K+1)}} \\
&=\seminorma{u}^2_{1, D_{d(K+1)}} 
+ \sum_{\substack{T \subset D_{d(K+1)} \\ x_i   \in  T} }\seminorma{I_{\T} u}^2_{1,T} 
+ \sum_{\substack{T \subset D_{d(K+1)} \\ x_i \notin T} }\seminorma{I_{\T} u}^2_{1,T}\\
&=: B_1 + B_2 + B_3
\end{align*}
From~\eqref{C-singamma} and lemma~\ref{dist}, we obtain:
\begin{align*}
B_1 &=\seminorma{u}^2_{1, D_{d(K+1)}} \leq \seminorma{u}^2_{1,
B(x_i,c2^{-(K+1)})} \leq 2\pi C \int_{0}^{c2^{-(K+1)}} r^{2(\gamma -
1)} r^{d-1} dr
\\
&= 2 \pi C \int_{0}^{c2^{-(K+1)}} r^{2 \gamma +d - 3} dr \simeq C 
2^{-(K+1) (2\gamma +d-2)}.
\end{align*}
For the term $B_2$ we use the fact that on a reference element $\hat T$, 
$\big|\widehat{\IT u}\big|_{1,\hat T} \lesssim \big\|\widehat{\IT u}\big\|_{L^\infty (\hat T)} = \big\|\IT u \big\|_{L^\infty (T)} $. By~\eqref{C-singamma}, if $ x_i \in T $, and $ T \subset D_{d(K+1)} $, $ \big\|\IT u \big\|_{L^\infty (T)} \leq C h_T^{\gamma} $. A proper scaling leads to
\begin{align*}
B_2 &= \sum_{\substack{T \subset D_{d(K+1)} \\ x_i \in T} }\seminorma{I_{\T} u}^2_{1,T} 
\approx \sum_{\substack{T \subset D_{d(K+1)} \\ x_i \in T} } h_T^{d-2} \big|\widehat{(\IT u)\vert_T}\big|_{1,\hat T}^{2} \lesssim C \sum_{\substack{T \subset D_{d(K+1)} \\ x_i \in T} }h_T^{2 \gamma + d -2} \\
&\leq \# \{ T \subset D_{d(K+1)} : x_i \in T \} \, |D_{d(K+1)}|^{\frac{2\gamma+d-2}{d}} . \\
\end{align*}
Since for these $T$'s, $r_T = 0$, lemma~\ref{size} leads to
\[
B_2 \lesssim  \# \{ T \subset D_{d(K+1)} : x_i \in T \} (2^{-d (K+1)})^{\frac{2\gamma+d-2}{d}} \lesssim 2^{-(K+1) (2 \gamma + d -2)},
\]
where we have used that the number of elements which have
$x_i$ as a vertex is bounded by a constant depending only on mesh regularity.

The term $B_3$ can be bounded using the fact that if $\dist(x_i,T) > 0$, then, by lemma~\ref{dist}, $ \dist(x_i,T) \simeq |x-x_i| $ $ \forall x \in T $ and thus~\eqref{C-singamma} yields
\begin{equation*}
|\nabla I_{\T} u(x)| \lesssim C \dist(x_i,T)^{\gamma -1} \lesssim
C |x-x_i|^{\gamma -1} \ \ \forall x \in T,
\end{equation*}
which implies that 
$
\int_T |\nabla I_{\T} u|^2 \lesssim C \int_T |x-x_i|^{2(\gamma - 1)} \,dx
$, 
and consequently
\begin{align*}
B_3 &= \sum_{\substack{T \subset D_{d(K+1)} \\ x_i \notin T }} \int_T | \nabla
I_{\T} u |^2 \lesssim C \int_{D_{d(K+1)}} |x-x_i|^{2(\gamma - 1)} dx \\
&\lesssim C \int_{0}^{c2^{-(K+1)}} r^{2 (\gamma - 1)} r^{d-1} \, dr \simeq
C 2^{-(K+1)(2\gamma +d-2)}.
\end{align*}
Combining the three estimates for $B_1$, $B_2$ and $B_3$ we obtain the following bound for the second term of~\eqref{main}:
\begin{equation}\label{second}
\seminorma{u - I_{\T} u}^2_{1,D_{d(K+1)}} \lesssim C  2^{-(K+1)(2\gamma +d-2)} \leq C \delta^{2p+d}.
\end{equation}
Using the usual estimates for the Lagrange interpolator and the fact that $ u\vert_T \in
H^{p+1}(T) $, $ \forall T \subset \Omega \backslash D_{d(K+1)} $ (see remark~\ref{R:main-result}), we can bound the first term of~(\ref{main})
by:	
\begin{equation}\label{first}
\sum_{\ell =0}^{d(K+1)-1} \sum_{T \subset D_{\ell}} \seminorma{u -
I_{\T} u}^2_{1,T} \lesssim \sum_{\ell =0}^{d(K+1)-1} \sum_{T \subset
D_{\ell}} h_T^{2p} \norma{D^{p+1}u}^{2}_{L^2(T)}.
\end{equation}
 Finally, by \eqref{C-singamma}, if $ x \in T $, $ |D^{p+1}u(x)| \leq C |x-x_i|^{\gamma-p-1} $, and thus $ \norma{D^{p+1}u}_{L^2(T)^{2}} \leq C \dist(x_i,T)^{2(\gamma-p-1)} h_T^d $, by  lemma
\ref{size}, $ h_T <  \delta 2^{\frac{2\ell
( \gamma - p -1 )}{d(2p+d)}} $ if $T \in D_{\ell}$, and again by lemma~\ref{dist}, we have that
\begin{align*}
\sum_{\ell =0}^{d(K+1)-1} \sum_{T \subset D_{\ell}} \seminorma{u - I_{\T} u}^2_{1,T} 
&\lesssim \sum_{\ell =0}^{d(K+1)-1} \sum_{T \subset D_{\ell}} h_T^{2p}
\norma{D^{p+1}u}^{2}_{L^2(T)} \\
&\lesssim  C \sum_{\ell =0}^{d(K+1)-1} \sum_{T \subset D_{\ell}}
\dist(x_i,T)^{2(\gamma -p-1)} h_T^{2p+d} \\
&\lesssim C \sum_{\ell =0}^{d(K+1)-1} \sum_{T \subset D_{\ell}} 2^{ -\frac{2 \ell
(\gamma - p -1)}{d}} \,  \delta^{2p+d} \, 2^{\frac{2 \ell (\gamma - p -1)}{d}}
\\
&\leq C \delta^{2p+d} \sum_{\ell=0}^{d(K+1)-1} \# D_{\ell}  
= C \delta^{2p+d} ( \# \T ) 
\end{align*}
Summing up, by~\eqref{main}, \eqref{second} and~\eqref{first}, and by lemma~\ref{complexity}
\begin{align*}
\seminorma{u-I_{\T}u}^2_{1, \Omega} 
&\lesssim C \delta^{2p+d} ( \# \T )  = C \delta^{2p+d} (
(\# \T - \# \T_0) + \# \T_0  ) \\
&\lesssim C \delta^{2p+d} (\delta^{-d} + \# \T_0 )\\
&\lesssim C \delta^{2p}
\lesssim C ( \# \T - \# \T_0 )^{-\frac{2p}{d}} ,
\end{align*}
where we have used that $\delta$ was chosen  so that $ \# \T_0 \leq \delta^{-d} $.
\end{proof}

\subsection{Proof of Main Result}\label{s:main-proof}

\begin{proof}[Proof of Theorem~\ref{T:error}]
Using the estimates of theorems~\ref{T:est-reg}
and~\ref{T:est-sing} we obtain:
\begin{align*}
\inf_{u_{\T} \in \V_{\T}} \norma{u-u_{\T}}_{1, \Omega} &\lesssim \inf_{u_{\T} \in \V_{\T}} \seminorma{u-u_{\T}}_{1,
\Omega} \lesssim C \seminorma{u-I_{\T}u}_{1,
\Omega} \\
&= C \seminorma{\sum_{i=0}^{n} ( u_i -
I_{\T}u_i)}_{1, \Omega} \\
&\leq C \sum_{i=0}^{N} \seminorma{( u_i - I_{\T}u_i)}_{1, \Omega}
\lesssim C N \delta^{p},
\end{align*}
and then, using lemma~\ref{complexity}, we have that
\begin{equation*}
\inf_{u_{\T} \in \V_{\T}} \norma{u-u_{\T}}_{1, \Omega} \lesssim
C N ( \# \T - \# \T_0 )^{-\frac{p}{d}}.
\end{equation*}
\end{proof}

\begin{proof}[Proof of Theorem~\ref{T:main-result}]
This is a corollary of theorem~\ref{T:error}. It is sufficient to
choose $\varepsilon = \mathbf{A}_1 \delta^{p}$. This implies the claim
for $ \varepsilon $ small enough, which immediately implies the
result for all $ \varepsilon > 0 $.
\end{proof}

\begin{remark} \textbf{Red-Green refinement}.
 Regarding the other well-known algorithm for adaptive mesh refinement in two dimensions, namely, the so called red-green refinement, the main result presented in this article is still open. However, the algorithm stated here can still be used for the construction of the quasi-optimal mesh, with obvious modifications due to the fact that a red subdivision splits the elements into four sub-elements instead of two. The only remaining issue that needs to be solved is to determine if a complexity result bounding the \emph{spreading} of refined elements, similar to theorem~\ref{T:bdd} holds.
\end{remark}





\subsection*{Acknowledgements} The authors want to thank Ricardo H.\ Nochetto for many interesting discussions and suggestions. They also want to express their gratitude to the anonymous referees, who, through their comments and suggestions helped us to substantially improve the manuscript.


\end{document}